\documentclass[11pt,a4paper,leqno]{amsart}
\usepackage[english]{babel}

\usepackage[latin1]{inputenc}
\usepackage[T1]{fontenc}
\usepackage{amsfonts}
\usepackage{amsmath}
\usepackage{amssymb}
\usepackage{eurosym}
\usepackage{mathrsfs}
\usepackage{palatino}
\usepackage{color}
\usepackage{esint}
\usepackage{url}
\usepackage{verbatim}
\allowdisplaybreaks[4]
\usepackage{enumerate}

\newcommand{\R}{\mathbb{R}}
\newcommand{\C}{\mathbb{C}}

\newcommand{\N}{\mathbb{N}}

\newcommand{\bla}{\big \langle}
\newcommand{\bra}{\big \rangle}

\numberwithin{equation}{section}

\newcommand{\ud}[0]{\,\mathrm{d}}

\newcommand{\esssup}[0]{\operatornamewithlimits{ess\,sup}}

\newcommand{\BMO}[0]{\operatorname{BMO}}
\newcommand{\bmo}[0]{\operatorname{bmo}}

\renewcommand{\Re}[0]{\operatorname{Re}}

\newcommand{\ch}[0]{\operatorname{ch}}

\newcommand{\calD}[0]{\mathcal{D}}

\newcommand{\wt}[1]{{\widetilde{#1}}}

\swapnumbers
\theoremstyle{plain}
\newtheorem{thm}[equation]{Theorem}
\newtheorem{lem}[equation]{Lemma}
\newtheorem{prop}[equation]{Proposition}

\theoremstyle{definition}

\theoremstyle{remark}
\newtheorem{rem}[equation]{Remark}

\author{Kangwei Li}
\address[K.L.]{BCAM (Basque Center for Applied Mathematics), Alameda de Mazarredo 14, 48009 Bilbao, Spain}
\email{kangwei.nku@gmail.com}

\author{Henri Martikainen}
\address[H.M.]{Department of Mathematics and Statistics, University of Helsinki, P.O.B. 68, FI-00014 University of Helsinki, Finland}
\email{henri.martikainen@helsinki.fi}

\author{Emil Vuorinen}
\address[E.V.]{Centre for Mathematical Sciences, University of Lund, P.O.B. 118, 22100 Lund, Sweden}
\email{j.e.vuorin@gmail.com}

\title{Bloom type upper bounds in the product BMO setting}

\makeatletter
\@namedef{subjclassname@2010}{%
  \textup{2010} Mathematics Subject Classification}
\makeatother


\begin{document}

\begin{abstract}
We prove some Bloom type estimates in the product BMO setting. More specifically,
for a bounded singular integral $T_n$ in $\R^n$ and a bounded singular integral $T_m$ in $\R^m$ we prove that
$$
\| [T_n^1, [b, T_m^2]] \|_{L^p(\mu) \to L^p(\lambda)} \lesssim_{[\mu]_{A_p}, [\lambda]_{A_p}} \|b\|_{\BMO_{\textup{prod}}(\nu)},
$$
where  $p \in (1,\infty)$, $\mu, \lambda \in A_p$ and $\nu := \mu^{1/p}\lambda^{-1/p}$ is the Bloom weight. Here $T_n^1$ is $T_n$ acting on the first variable,
$T_m^2$ is $T_m$ acting on the second variable, $A_p$ stands for the bi-parameter weights of $\R^n \times \R^m$ and
$\BMO_{\textup{prod}}(\nu)$ is a weighted product BMO space.
\end{abstract}

\maketitle

\section{Introduction}
Let $\mu$ and $\lambda$ be two general Radon measures in $\R^n$.
A two-weight problem asks for a characterisation of the boundedness $T \colon L^p(\mu) \to L^p(\lambda)$, where
$T$ is, for instance, a singular integral operator (SIO). 
Singular integral operators take the form
$$
  Tf(x)=\int_{\R^n}K(x,y)f(y)\ud y,
$$
where different assumptions on the {\em kernel} $K$ lead to important classes of linear transformations arising across pure and applied analysis.
There exists a two-weight characterisation for the Hilbert transform $T =H$, where $K(x,y) = 1/(x-y)$,
 by Lacey \cite{La1} and Lacey, Sawyer, Uriarte-Tuero and Shen \cite{LSUS} (see also Hyt\"onen \cite{Hy3}).

In a Bloom type variant of the two-weight question we require that $\mu$ and $\lambda$ are Muckenhoupt $A_p$ weights and that the problem
involves a function $b$ that can be taken to be in some appropriate weighted $\BMO$ space $\BMO(\nu)$ formed using the Bloom weight $\nu := \mu^{1/p}\lambda^{-1/p} \in A_2$.
The presence of the function $b$ lands us naturally to the commutator setting.
Coifman--Rochberg--Weiss \cite{CRW} showed that the commutator $[b,T]f := bTf - T(bf)$ satisfies
$$
\|b\|_{\BMO} \lesssim \|[b,T]\|_{L^p \to L^p} \lesssim \|b\|_{\BMO}, \qquad p \in (1,\infty),
$$
for a wide class of singular integrals $T$. These estimates are called the commutator lower bound and the commutator upper bound, respectively. Both estimates are non-trivial and proved quite differently. The upper bound should hold for all bounded SIOs, while the lower bound requires some suitable non-degenaracy.
Given some operator $A^b$, the definition of which depends naturally on some function $b$,
the Bloom type questions concerns the validity of the estimate
$$
\|A^b\|_{L^p(\mu) \to L^p(\lambda)} \lesssim_{[\mu]_{A_p}, [\lambda]_{A_p}} \|b\|_{\BMO(\nu)}.
$$
When studied in the natural commutator setting we may also, for suitable singular integrals $T$, hope to prove the appropriate lower bound, such as
$$
\|b\|_{\BMO(\nu)} \lesssim_{[\mu]_{A_p}, [\lambda]_{A_p}} \|[b,T]\|_{L^p(\mu) \to L^p(\lambda)}.
$$
In the Hilbert transform case $T = H$ Bloom \cite{Bl} proved such a two-sided estimate.

 A renewed interest on the Bloom type estimates started from the recent works of
Holmes--Lacey--Wick \cite{HLW, HLW2}. They proved Bloom's upper bound for general bounded SIOs in $\R^n$ using modern proof techniques, and considered
the lower bound in the Riesz case.
Lerner--Ombrosi--Rivera-R\'ios \cite{LOR1} further refined these results
using sparse domination. An iterated commutator of the form $[b, [b,T]]$ is studied by Holmes--Wick \cite{HW}, when $b \in \BMO \cap \BMO(\nu)$. It was also possible to prove this iterated case from the first order case via the so-called Cauchy integral trick of Coifman--Rochberg--Weiss \cite{CRW}, see
Hyt\"onen \cite{Hy4}. However, this works precisely because it is also assumed that $b \in \BMO$. 
It turns out that  $\BMO \cap \BMO(\nu)$ is not optimal.  A fundamentally improved iterated case is by Lerner--Ombrosi--Rivera-R\'ios \cite{LOR2}. There $b \in \BMO(\nu^{1/2}) \supsetneq \BMO \cap \BMO(\nu)$, and this is a characterisation as lower bounds also hold by \cite{LOR2}. In the recent paper \cite{Hy5} by Hyt\"onen lower bounds with weak non-degeneracy assumptions on
$T$ (or its kernel $K$) are shown. On the other hand, multilinear Bloom type inequalities were initiated by Kunwar--Ou \cite{KO}. 

We are interested in Bloom type inequalities in the bi-parameter setting.
Classical {\em one-parameter} kernels are ``singular'' (involve ``division by zero'') exactly when $x=y$. In contrast, the {\em multi-parameter} theory is concerned with kernels whose singularity is spread over the union of all hyperplanes of the form $x_i=y_i$, where $x,y\in\R^d$ are written as $x=(x_i)_{i=1}^t\in\R^{d_1}\times\cdots\times\R^{d_t}$ for a fixed partition $d=d_1+\ldots+d_t$. The bi-parameter case $d=d_1+d_2=n+m$ is already representative of many of the challenges arising in this context. The prototype example is $1/[(x_1-y_1)(x_2-y_2)]$, the product of Hilbert kernels in both coordinate directions of $\R^2$, but general two-parameter kernels are neither assumed to be of the product nor of the convolution form. 
 A general bi-parameter SIO is a linear operator with various kernel representations depending
whether we have separation in both $\R^n$ and $\R^m$ (a full kernel representation), or just in $\R^n$ or $\R^m$ (a partial kernel representation).
See the bi-parameter representation theorem \cite{Ma1} by one of us. This is the modern viewpoint but deals with the same class of SIOs as originally introduced by Journ\'e \cite{Jo2}  (see Grau de la Herr\'an \cite{Grau}).

In general, there is a vast difference between the techniques that are required in the bi-parameter (or more generally multi-parameter) setting and in the one-parameter theory. For example, there is only one sparse domination paper in this setting, namely Barron--Pipher \cite{BP}, but the sparse domination is quite a bit more restricted
than in the one-parameter setting.

Bloom theory in the bi-parameter setting is very recent. Bloom type upper bounds for commutators of the form $[b,T]$ are valid for bi-parameter Calder\'on--Zygmund operators $T$ and they are established using the bi-parameter dyadic representation theorem \cite{Ma1} and carefully estimating the resulting commutators of various dyadic model operators (DMOs). Here we define a bi-parameter Calder\'on--Zygmund operator (CZO) to be a bi-parameter SIO such that $T1, T^*1, T_1(1), T_1^*(1) \in \BMO_{\textup{prod}}$, and certain additional weak testing conditions hold. Here $T_1$ is a \emph{partial adjoint} of $T$ in the first slot, $\langle T_1(f_1 \otimes f_2), g_1 \otimes g_2 \rangle = \langle T(g_1 \otimes f_2), f_1 \otimes g_2 \rangle$, and $\BMO_{\textup{prod}}$ is the product BMO of Chang and Fefferman \cite{CF1, CF2}. Holmes--Petermichl--Wick \cite{HPW} proved the first bi-parameter Bloom type estimate
$$
\|[b,T]\|_{L^p(\mu) \to L^p(\lambda)} \lesssim_{[\mu]_{A_p}, [\lambda]_{A_p}} \|b\|_{\bmo(\nu)}.
$$
Here $A_p$ stands for bi-parameter weights (replace cubes by rectangles in the usual definition) and $\bmo(\nu)$ is the weighted little $\BMO$ space defined using the norm
$$
\|b\|_{\bmo(\nu)} := \sup_{R} \frac{1}{\nu(R)} \int_R |b - \langle b \rangle_R|, \textup{ where } \langle b \rangle_R = \frac{1}{|R|} \int_R b, \,\, \nu(R) = \int_R \nu, 
$$
and the supremum is over all rectangles $R = I \times J \subset \R^n \times \R^m$. We refer to these types of commutators as little $\BMO$ commutators. 

In the recent paper \cite{LMV2} we reproved the result of \cite{HPW} in an efficient way based on improved commutator decompositions from
our bilinear bi-parameter theory \cite{LMV1}. The clear structure of our proof allowed us to also handle the iterated little BMO commutator case and to prove the upper bound
\begin{equation}\label{eq:LMVIterated}
\| [b_k,\cdots[b_2, [b_1, T]]\cdots]\|_{L^p(\mu) \to L^p(\lambda)} \lesssim_{[\mu]_{A_p}, [\lambda]_{A_p}} \prod_{i=1}^k\|b_i\|_{\bmo(\nu^{\theta_i})}, \qquad \sum_{i=1}^k\theta_i=1.
\end{equation}

There is also a different type of commutator, which is equally fundamental in the bi-parameter setting, but
which we have not yet discussed. \emph{In this paper we are interested in Bloom type
question for this commutator} -- the main point is that now the right space is not the little $\bmo$ but the much harder product $\BMO$.
If $T_n$ and $T_m$ are one-parameter CZOs (bounded SIOs) in $\R^n$ and $\R^m$, respectively, then for the $L^p \to L^p$ norm of the commutator
$$
[T_n^1, [b, T_m^2]]f = T_n^1(b T_m^2 f) - T_n^1T_m^2(bf) - b T_m^2 T_n^1 f + T_m^2(bT_n^1f),
$$
where $T_n^1 f(x) = T_n(f(\cdot, x_2))(x_1)$,
the right object is $b \in \BMO_{\textup{prod}}(\R^{n+m})$. We refer to these commutators as product $\BMO$ type commutators.
Classical references for commutators of both type in the Hilbert transform case include Ferguson--Sadosky \cite{FeSa} and the groundbreaking paper Ferguson--Lacey \cite{FL}.
In particular, \cite{FL} contains the extremely difficult lower bound estimate
$$
\|b\|_{\BMO_{\textup{prod}}(\R^{2})} \lesssim \| [H^1, [b, H^2]]\|_{L^2 \to L^2}.
$$
See also Lacey--Petermichl--Pipher--Wick \cite{LPPW, LPPW2, LPPW3} for the higher dimensional Riesz setting and applications to div-curl lemmas.
The corresponding upper bound is not easy either -- even for special operators such as the Riesz transforms. However, in \cite{DaO} Dalenc and Ou proved that for all CZOs
$$
\| [T_n^1, [b, T_m^2]] \|_{L^p \to L^p} \lesssim \|b\|_{\BMO_{\textup{prod}}}
$$
using the modern approach via the one-parameter representation theorem of Hyt\"onen \cite{Hy}. 

However, no Bloom type estimates have been considered. In this paper we prove the Dalenc--Ou \cite{DaO} type bound in the Bloom setting -- i.e., we prove Bloom type upper bounds
for product BMO commutators. The following is our main result.
\begin{thm}
Let $T_n$ and $T_m$ be one-parameter CZOs in $\R^n$ and $\R^m$, respectively, and let $p \in (1,\infty)$,
$\mu, \lambda \in A_p$ and $\nu := \mu^{1/p}\lambda^{-1/p}$. We have the quantitative estimate
$$
\| [T_n^1, [b, T_m^2]] \|_{L^p(\mu) \to L^p(\lambda)} \lesssim_{[\mu]_{A_p}, [\lambda]_{A_p}} \|b\|_{\BMO_{\textup{prod}}(\nu)}.
$$
\end{thm}
\noindent Here $\BMO_{\textup{prod}}(\nu)$ is a weighted product BMO space as defined at least in \cite{HPW, LMV2}.
We note that $\bmo(\nu) \subset \BMO_{\textup{prod}}(\nu)$ (a proof can be found at least in \cite{LMV2}).

This product BMO Bloom estimate is proved in Section \ref{sec:ProofBloom}. While we were inspired by our previous paper
\cite{LMV2} dealing with the new approach to little BMO commutators, a quite different take on things is required by the current product BMO setting. The proof idea is outlined
in the beginning of Section \ref{sec:ProofBloom}. 

It is probably
possible to add more singular integrals to this Bloom bound and/or consider multi-parameter CZOs here, but
we content with the most fundamental case here. However, we still mention the following. If we consider multi-parameter singular integrals in the product $\BMO$ type commutators, then some kind of little product $\BMO$
assumptions concerning $b$ are the right thing.
Suppose e.g. that $T_1$ and $T_2$ are bi-parameter CZOs in $\R^{n_1} \times \R^{n_2}$ and $\R^{n_3} \times \R^{n_4}$, respectively.
Then according to Ou--Petermich--Strouse \cite{OPS} and Holmes--Petermichl--Wick \cite{HPW} we have
\begin{align*}
\|[T_1, [b, T_2]]f&\|_{L^2(\prod_{i=1}^4 \R^{n_i})} \lesssim \max\big( \sup_{x_2, x_4} \|b(\cdot, x_2, \cdot, x_4)\|_{\BMO_{\operatorname{prod}}},
 \sup_{x_2, x_3} \|b(\cdot, x_2, x_3, \cdot)\|_{\BMO_{\operatorname{prod}}}, \\
 & \sup_{x_1, x_4} \|b(x_1, \cdot, \cdot, x_4)\|_{\BMO_{\operatorname{prod}}},
  \sup_{x_1, x_3} \|b(x_1, \cdot, x_3, \cdot)\|_{\BMO_{\operatorname{prod}}} \big) \|f\|_{L^2(\prod_{i=1}^4 \R^{n_i})},
\end{align*}
where instead of $T_1$ we could also similarly as above write $T_1^{1,2}$ highlighting the fact that here it acts on the first two variables. 
In general, the field of multi-parameter commutator estimates is again very active -- we mention e.g. that recently the commutators of multi-parameter flag singular integrals were investigated by Duong--Li--Ou--Pipher--Wick \cite{DLOPW}.

The product BMO Bloom type upper bound is the main contribution of this paper. However, we also complement our previous paper \cite{LMV2} and \eqref{eq:LMVIterated} by giving
an easy little BMO iterated commutator lower bound proof using the so-called median method, previously used in the one-parameter setting in \cite{LOR2, Hy5} (see also \cite{LMV1}). In \cite{LMV2} we only recorded the following remark regarding the Estimate \eqref{eq:LMVIterated}.
Choosing $b_1 = \cdots = b_k = b$ and $\theta_1 = \cdots = \theta_k = 1/k$ in \eqref{eq:LMVIterated}
we get a bi-parameter analog of \cite{LOR2}, while choosing $\theta_1 = 1$ (and the rest zero) we get analogs of \cite{HW, Hy4}. However, the first is the better choice as $\bmo(\nu^{1/k}) \supset \bmo \cap \bmo(\nu)$. Indeed, similarly as in the one-parameter case
\cite{LOR2}, this is seen by using that $\langle \nu \rangle_{R}^{\theta} \lesssim_{[\nu]_{A_2}} \langle \nu^{\theta} \rangle_{R}$ for all $\theta \in (0,1)$ and rectangles $R$ (this estimate follows from \cite[Theorem 2.1]{CN} by iteration). In this paper we actually prove the lower bound showing the optimality of $\bmo(\nu^{1/k})$. The details are given quickly in Section \ref{sec:LowerBounds}.

\subsection*{Acknowledgements}
K. Li was supported by Juan de la Cierva - Formaci\'on 2015 FJCI-2015-24547, by the Basque Government through the BERC
2018-2021 program and by Spanish Ministry of Economy and Competitiveness
MINECO through BCAM Severo Ochoa excellence accreditation SEV-2017-0718
and through project MTM2017-82160-C2-1-P funded by (AEI/FEDER, UE) and
acronym ``HAQMEC''.

H. Martikainen was supported by the Academy of Finland through the grants 294840 and 306901, and by the three-year research grant 75160010 of the University of Helsinki.
He is a member of the Finnish Centre of Excellence in Analysis and Dynamics Research.

E. Vuorinen was supported by the Academy of Finland through the grant 306901, by the Finnish Centre of Excellence in Analysis and Dynamics Research, and by
Jenny and Antti Wihuri Foundation.

\section{Definitions and preliminaries}
\subsection{Basic notation}
We denote $A \lesssim B$ if $A \le CB$ for some constant $C$ that can depend on the dimension of the underlying spaces, on integration exponents, and on various other constants appearing in the assumptions. We denote $A \sim B$ if $B \lesssim A \lesssim B$.

We work in the bi-parameter setting in the product space $\R^{n+m} = \R^n \times \R^m$.
In such a context $x = (x_1, x_2)$ with $x_1 \in \R^n$ and $x_2 \in \R^m$.
We often take integral pairings with respect to one of the two variables only:
If $f \colon \R^{n+m} \to \C$ and $h \colon \R^n \to \C$, then $\langle f, h \rangle_1 \colon \R^{m} \to \C$ is defined by
$
\langle f, h \rangle_1(x_2) = \int_{\R^n} f(y_1, x_2)h(y_1)\ud y_1.
$

\subsection{Dyadic notation, Haar functions and martingale differences}
We denote a dyadic grid in $\R^n$ by $\calD^n$ and a dyadic grid in $\R^m$ by $\calD^m$. If $I \in \calD^n$, then $I^{(k)}$ denotes the unique dyadic cube $S \in \calD^n$ so that $I \subset S$ and $\ell(S) = 2^k\ell(I)$. Here $\ell(I)$ stands for side length. Also, $\text{ch}(I)$ denotes the dyadic children of $I$, i.e., $I' \in \ch(I)$ if $I' \in \calD^n$, $I' \subset I$ and $\ell(I') = \ell(I)/2$. We often write $\calD = \calD^n \times \calD^m$.

When $I \in \calD^n$ we denote by $h_I$ a cancellative $L^2$ normalised Haar function. This means the following.
Writing $I = I_1 \times \cdots \times I_n$ we can define the Haar function $h_I^{\eta}$, $\eta = (\eta_1, \ldots, \eta_n) \in \{0,1\}^n$, by setting
\begin{displaymath}
h_I^{\eta} = h_{I_1}^{\eta_1} \otimes \cdots \otimes h_{I_n}^{\eta_n}, 
\end{displaymath}
where $h_{I_i}^0 = |I_i|^{-1/2}1_{I_i}$ and $h_{I_i}^1 = |I_i|^{-1/2}(1_{I_{i, l}} - 1_{I_{i, r}})$ for every $i = 1, \ldots, n$. Here $I_{i,l}$ and $I_{i,r}$ are the left and right
halves of the interval $I_i$ respectively. The reader should carefully notice that $h_I^0$ is the non-cancellative Haar function for us and that in some other papers a different convention is used. If $\eta \in \{0,1\}^n \setminus \{0\}$ the Haar function is cancellative: $\int h_I^{\eta} = 0$. We usually suppress the presence of $\eta$
and simply write $h_I$ for some $h_I^{\eta}$, $\eta \in \{0,1\}^n \setminus \{0\}$. Then $h_Ih_I$ can stand for $h_I^{\eta_1} h_I^{\eta_2}$, but we always treat
such a product as a non-cancellative function i.e. use only its size.

For $I \in \calD^n$ and a locally integrable function $f\colon \R^n \to \C$, we define the martingale difference
$$
\Delta_I f = \sum_{I' \in \textup{ch}(I)} \big[ \bla f \bra_{I'} -  \bla f \bra_{I} \big] 1_{I'}.
$$
Here $\bla f \bra_I = \frac{1}{|I|} \int_I f$. We also write $E_I f = \bla f \bra_I 1_I$.
Now, we have $\Delta_I f = \sum_{\eta \ne 0} \langle f, h_{I}^{\eta}\rangle h_{I}^{\eta}$, or suppressing the $\eta$ summation, $\Delta_I f = \langle f, h_I \rangle h_I$, where $\langle f, h_I \rangle = \int f h_I$. A martingale block is defined by
$$
\Delta_{K,i} f = \mathop{\sum_{I \in \calD^n}}_{I^{(i)} = K} \Delta_I f, \qquad K \in \calD^n, i \in \N.
$$

Next, we define bi-parameter martingale differences. Let $f \colon \R^n \times \R^m \to \C$ be locally integrable.
Let $I \in \calD^n$ and $J \in \calD^m$. We define the martingale difference
$$
\Delta_I^1 f \colon \R^{n+m} \to \C, \Delta_I^1 f(x) := \Delta_I (f(\cdot, x_2))(x_1).
$$
Define $\Delta_J^2f$ analogously, and also define $E_I^1$ and $E_J^2$ similarly.
We set
$$
\Delta_{I \times J} f \colon \R^{n+m} \to \C, \Delta_{I \times J} f(x) = \Delta_I^1(\Delta_J^2 f)(x) = \Delta_J^2 ( \Delta_I^1 f)(x).
$$
Notice that $\Delta^1_I f = h_I \otimes \langle f , h_I \rangle_1$, $\Delta^2_J f = \langle f, h_J \rangle_2 \otimes h_J$ and
$ \Delta_{I \times J} f = \langle f, h_I \otimes h_J\rangle h_I \otimes h_J$ (suppressing the finite $\eta$ summations).
Martingale blocks are defined in the natural way
$$
\Delta_{K \times V}^{i, j} f  =  \sum_{I\colon I^{(i)} = K} \sum_{J\colon J^{(j)} = V} \Delta_{I \times J} f = \Delta_{K,i}^1( \Delta_{V,j}^2 f) = \Delta_{V,j}^2 ( \Delta_{K,i}^1 f).
$$

\subsection{Weights}
A weight $w(x_1, x_2)$ (i.e. a locally integrable a.e. positive function) belongs to bi-parameter $A_p(\R^n \times \R^m)$, $1 < p < \infty$, if
$$
[w]_{A_p(\R^n \times \R^m)} := \sup_{R} \bla w \bra_R \bla w^{1-p'} \bra_R^{p-1} < \infty,
$$
where the supremum is taken over $R = I \times J$, where $I \subset \R^n$ and $J \subset \R^m$ are cubes
with sides parallel to the axes (we simply call such $R$ rectangles).
We have
$$
[w]_{A_p(\R^n\times \R^m)} < \infty \textup { iff } \max\big( \esssup_{x_1 \in \R^n} \,[w(x_1, \cdot)]_{A_p(\R^m)}, \esssup_{x_2 \in \R^m}\, [w(\cdot, x_2)]_{A_p(\R^n)} \big) < \infty,
$$
and that $\max\big( \esssup_{x_1 \in \R^n} \,[w(x_1, \cdot)]_{A_p(\R^m)}, \esssup_{x_2 \in \R^m}\, [w(\cdot, x_2)]_{A_p(\R^n)} \big) \le [w]_{A_p(\R^n\times \R^m)}$, while
the constant $[w]_{A_p}$ is dominated by the maximum to some power.
Of course, $A_p(\R^n)$ is defined similarly as $A_p(\R^n \times \R^m)$ -- just take the supremum over cubes $Q$. For the basic theory
of bi-parameter weights consult e.g. \cite{HPW}.

\subsection{Square functions and maximal functions}
Given $f \colon \R^{n+m} \to \C$ and $g \colon \R^n \to \C$ we denote the dyadic maximal functions
by
$$
M_{\calD^n}g := \sup_{I \in \calD^n} \frac{1_I}{|I|}\int_I |g|\, \textup{ and } \,
M_{\calD} f:= \sup_{R \in \calD}  \frac{1_R}{|R|}\int_R |f|.
$$
We also set $M^1_{\calD^n} f(x_1, x_2) =  M_{\calD^n}(f(\cdot, x_2))(x_1)$. The operator $M^2_{\calD^m}$ is defined similarly.

Define the square functions
$$
S_{\calD} f = \Big( \sum_{R \in \calD}  |\Delta_{R} f|^2 \Big)^{1/2}, \,\, S_{\calD^n}^1 f =  \Big( \sum_{I \in \calD^n}  |\Delta_I^1 f|^2 \Big)^{1/2}\, \textup{ and } \,
S_{\calD^m}^2 f =  \Big( \sum_{J \in \calD^m} |\Delta_J^2 f|^2 \Big)^{1/2}.
$$
Define also
$$
S_{\calD, M}^1 f = \Big( \sum_{I \in \calD^n} \frac{1_I}{|I|} \otimes [M_{\calD^m} \langle f, h_I \rangle_1]^2 \Big)^{1/2}
\, \textup{ and } \, S_{\calD, M}^2 f = \Big( \sum_{J \in \calD^m} [M_{\calD^n} \langle f, h_J \rangle_2]^2 \otimes \frac{1_J}{|J|}\Big)^{1/2}.
$$ 
We record the following standard estimates, which are used repeatedly below in an implicit manner. Some
similar estimates that follow from the ones below are also used.
\begin{lem}\label{lem:standardEst}
For $p \in (1,\infty)$ and $w \in A_p = A_p(\R^n \times \R^m)$ we have the weighted square function estimates
$$
\| f \|_{L^p(w)}
 \sim_{[w]_{A_p}} \| S_{\calD} f\|_{L^p(w)} 
\sim_{[w]_{A_p}}  \| S_{\calD^n}^1 f  \|_{L^p(w)}
\sim_{[w]_{A_p}} \| S_{\calD^m}^2 f  \|_{L^p(w)}.
$$
Moreover, for $p, s \in (1,\infty)$ we have the Fefferman--Stein inequality
$$
\Big\| \Big( \sum_j |M f_j |^s \Big)^{1/s} \Big\|_{L^p(w)} \lesssim_{[w]_{A_p}} \Big\| \Big( \sum_{j} | f_j |^s \Big)^{1/s} \Big\|_{L^p(w)}.
$$
Here $M$ can e.g. be $M_{\calD^n}^1$ or $M_{\calD}$. Finally, we have
$$
\| S_{\calD, M}^1 f\|_{L^p(w)} + \| S_{\calD, M}^2 f\|_{L^p(w)} \lesssim_{[w]_{A_p}} \|f\|_{L^p(w)}.
$$
\end{lem}

One easy way to show such estimates is to reduce to p = 2 via standard extrapolation. When p = 2 it is especially easy to use one-parameter results iteratively. See e.g. \cite{CMP, CWW} for one-parameter square function results and their history.

\subsection{BMO spaces}\label{ss:bmo}
Let $b \colon \R^{n+m} \to \C$ be locally integrable.

We define the weighted product BMO space. Given $\nu \in A_2(\R^n \times \R^m)$ set
$$
\|b\|_{\BMO_{\textup{prod}}(\nu, \calD)} := 
\sup_{\Omega} \Big( \frac{1}{\nu(\Omega)} \mathop{\sum_{R \in \calD}}_{R \subset \Omega} |\langle b, h_R\rangle|^2 \langle \nu \rangle_{R}^{-1} \Big)^{1/2},
$$
where $h_R := h_I \otimes h_J$ and the supremum is taken over those sets $\Omega \subset \R^{n+m}$ such that $|\Omega| < \infty$ and such that for every $x \in \Omega$ there exist
$R \in \calD$ so that $x \in R \subset \Omega$.
The non-dyadic product BMO space $\BMO_{\textup{prod}}(\nu)$ is defined using the norm defined by the supremum over all dyadic grids of
the above dyadic norms. 

We also say that $b$ belongs to the weighted dyadic little BMO space $\bmo(\nu, \calD)$ if
$$
\|b\|_{\bmo(\nu, \calD)} := \sup_{R \in \calD} \frac{1}{\nu(R)} \int_R |b - \langle b \rangle_R| < \infty.
$$
The non-dyadic variant $\bmo(\nu)$ has the obvious definition.
There holds $\bmo(\nu, \calD) \subset \BMO_{\textup{prod}}(\nu, \calD)$ (this is proved explicitly at least in \cite{LMV2}). We do not need this fact in this paper, however.

For a sequence $(a_I)_{I \in \calD^n}$ we define
$$
\| (a_I) \|_{\BMO(\calD^n)} = \sup_{I_0 \in \calD^n} \Big( \frac{1}{|I_0|} \sum_{I \subset I_0} |a_I|^2 \Big)^{1/2}.
$$

\section{Martingale difference expansions of products}\label{sec:marprod}
Let $\calD^n$ and $\calD^m$ be some fixed dyadic grids in $\R^n$ and $\R^m$, respectively, and write $\calD= \calD^n \times \calD^m$.
In what follows we sum over $I \in \calD^n$ and $J \in \calD^m$. In general, in this paper always $K, I, I_1, I_2 \in \calD^n$ and $V, J, J_1, J_2 \in \calD^m$. 

\subsubsection*{Paraproduct operators}
The product BMO type paraproducts are
\begin{align*}
A_1(b,f) &= \sum_{I, J} \Delta_{I \times J} b \Delta_{I \times J} f, \,\,
A_2(b,f) = \sum_{I, J} \Delta_{I \times J} b E_I^1\Delta_J^2 f, \\
A_3(b,f) &= \sum_{I, J} \Delta_{I \times J} b \Delta_I^1 E_J^2  f, \,\,
A_4(b,f) = \sum_{I, J} \Delta_{I \times J} b \bla f \bra_{I \times J}.
\end{align*}
The little BMO type paraproducts are
\begin{align*}
A_5(b,f) &= \sum_{I, J} E_I^1 \Delta_J^2 b \Delta_{I \times J} f, \,\,
A_6(b,f) = \sum_{I, J}  E_I^1 \Delta_J^2 b  \Delta_I^1 E_J^2  f, \\
A_7(b,f) &= \sum_{I, J} \Delta_I^1 E_J^2  b \Delta_{I \times J} f, \,\,
A_8(b,f) = \sum_{I, J}  \Delta_I^1 E_J^2 b E_I^1 \Delta_J^2 f.
\end{align*}
The  ``illegal'' bi-parameter paraproduct is
$$
W(b,f) = \sum_{I, J}  \bla b \bra_{I \times J} \Delta_{I \times J} f.
$$
Things make sense at least with the a priori assumptions that $b$ is bounded and $f$ is bounded and compactly supported.
For us $b \in \BMO_{\textup{prod}}(\nu)$ so that only the $A_1, \ldots A_4$ are good as standalone operators. 
If we would have $b \in \bmo(\nu) \subset \BMO_{\textup{prod}}(\nu)$, then also $A_5, \ldots, A_8$ would be good (but this is not the case here). So we can only rely on the following boundedness property. See \cite{HPW, LMV2}.
\begin{lem}\label{lem:basicAa}
Suppose $b\in \BMO_{\textup{prod}}(\nu)$, where $\nu=\mu^{1/p}\lambda^{-1/p}$, $\mu, \lambda\in A_p$ and $p \in (1,\infty)$. Then
for $i=1, \ldots, 4$ we have
\begin{align*}
\|A_i(b, \cdot)\|_{L^p(\mu)\rightarrow L^p(\lambda)} \lesssim_{[\mu]_{A_p}, [\lambda]_{A_p}} \|b\|_{\BMO_{\textup{prod}}(\nu)}.
\end{align*}
\end{lem}
We formally have
$$
bf = \sum_{i=1}^8 A_i(b, f) + W(b,f).
$$
When we write like this we say that we decompose $bf$ in the bi-parameter sense.
When we employ this decomposition in practice we have to form suitable differences of the problematic operators with some other terms to get something that is bounded with the product BMO assumption. The operator $W$ has a special role.

Next, we define operators related to one-parameter type decompositions.
Define the one-parameter paraproducts
$$
a^1_1(b,f) = \sum_I \Delta_I^1 b \Delta_I^1 f, \, \,
a^1_2(b,f) = \sum_I \Delta_I^1 b E_I^1 f.
$$
Define also the ``illegal'' one-parameter paraproduct
$$
w^1(b, f) = \sum_I E_I^1 b \Delta_I^1 f.
$$
The operators $a_i^1(b, \cdot)$ would be bounded if $b \in \bmo(\nu) \subset \BMO_{\textup{prod}}(\nu)$, but again this is not the case here.
The operators $a^2_1(b,f)$, $a^2_2(b,f)$ and $w^2(b,f)$ are defined analogously. We formally have
$$
bf = \sum_{i=1}^2 a^1_i(b,f) + w^1(b,f) = \sum_{i=1}^2 a^2_i(b,f) + w^2(b,f).
$$
In this case we say that we decomposed $bf$ in the one-parameter sense (either in $\R^n$ or $\R^m$).

\section{The Bloom type product $\BMO$ upper bound}\label{sec:ProofBloom}

\begin{thm}
Let $T_n$ and $T_m$ be one-parameter CZOs in $\R^n$ and $\R^m$, respectively, and let $p \in (1,\infty)$,
$\mu, \lambda \in A_p$ and $\nu := \mu^{1/p}\lambda^{-1/p}$. We have the quantitative estimate
$$
\| [T_n^1, [b, T_m^2]] \|_{L^p(\mu) \to L^p(\lambda)} \lesssim_{[\mu]_{A_p}, [\lambda]_{A_p}} \|b\|_{\BMO_{\textup{prod}}(\nu)}.
$$
\end{thm}
\begin{rem}
Here we are proving only the right quantitative bound, and prefer to
understand the inequality so that $b$ is nice to begin with -- we at least make the purely qualitative a priori assumption $b \in L^{\infty}(\R^{n+m})$.
\end{rem}
\begin{proof}
It is enough to prove the Bloom type
inequality for $ [U_n^1, [b, U_m^2]]$, where $U_n$ and $U_m$ are DMOs (dyadic model operators) appearing in the representation theorem \cite{Hy, Hy2}.
This means that $U_n \in \{S_n, \pi_n\}$, where $S_n$ is a dyadic shift in $\R^n$ and $\pi_n$ is a dyadic paraproduct in $\R^n$ (we will recall what these mean as we go).
We only have to maintain a polynomial dependence on the complexity of the shift. 

Let $f \colon \R^{n+m} \to \C$ be bounded and compactly supported.
How we expand various products of functions (in the bi-parameter sense, in the one-parameter sense, or not at all) depends on the structure of the model operators:
a cancellative Haar function in a correct position is required to tricker an expansion in the corresponding parameter.
In general, the model operators have the form
$$
U_n^1 f = \sum_{ \substack{K \\  I_i^{(k_i)} = K}} a_{K, (I_i)}  \tilde h_{I_2} \otimes \langle f, \tilde h_{I_1} \rangle_1, \,\,
U_m^2 f = \sum_{ \substack{V \\ J_j^{(v_j)} = V}} a_{V, (J_j)} \langle f, \tilde h_{J_1} \rangle_2 \otimes \tilde h_{J_2},
$$
where $k_1,k_2, v_1, v_2 \ge 0$, $K, I_1, I_2 \in \calD^n$, $V, J_1, J_2 \in \calD^m$, $a_{K, (I_i)}, a_{V, (J_j)}$ are appropriate constants
and $\tilde h_{I_i} \in \{h_{I_i}, 1_{I_i} / |I_i|\}$,  $\tilde h_{J_j} \in \{h_{J_j}, 1_{J_j} / |J_j|\}$. We will write
$$
[U_n^1, [b, U_m^2]]f = I - II - III + IV,
$$
where
$$
I = U_n^1 ( b U_m^2 f), \, II = U_n^1 U_m^2 (bf), \, III = b  U_m^2 U_n^1 f\, \textup{ and } \,IV = U_m^2 ( bU_n^1 f). 
$$
In all of the terms $I, \ldots, IV$ the appearing product $b \cdot U_m^2 f$, $b \cdot f$, $b \cdot U_m^2 U_n^1 f$ or $b \cdot U_n^1 f$, respectively, is expanded. We now explain which of the appearing
Haar functions determine the expansion strategy in each of the terms. For $I$ the determining Haar functions are
$\tilde h_{I_1}, \tilde h_{J_2}$, for $II$ they are $\tilde h_{I_1}, \tilde h_{J_1}$, for $III$ they are $\tilde h_{I_2}, \tilde h_{J_2}$ and for $IV$ they
are $\tilde h_{I_2}, \tilde h_{J_1}$. If both of the determining Haar functions are cancellative, we expand in the bi-parameter sense using the paraproducts $A_i(b, \cdot), W(b, \cdot)$. If one of them is
cancellative we expand in the corresponding parameter $k \in \{1,2\}$ using the paraproducts $a_i^k(b, \cdot), w^k(b, \cdot)$. In addition, when a non-cancellative Haar function appears we sometimes add and subtract an average to ease the upcoming expansions.
\subsection*{The case $[S_n^1, [b, S_m^2]]$}
We study the case that
$$
S_n^1 f = \sum_{ \substack{K \\  I_i^{(k_i)} = K}} a_{K, (I_i)}  h_{I_2} \otimes \langle f, h_{I_1} \rangle_1, \,\,
S_m^2 f = \sum_{ \substack{V \\ J_j^{(v_j)} = V}} a_{V, (J_j)} \langle f, h_{J_1} \rangle_2 \otimes h_{J_2}.
$$
Here $k_1,k_2, v_1, v_2 \ge 0$, $K, I_1, I_2 \in \calD^n$ and $V, J_1, J_2 \in \calD^m$. Only finitely many of the constants $a_{K, (I_i)}$ are non-zero
and
$$
|a_{K, (I_i)}| \le \frac{|I_1|^{1/2} |I_2|^{1/2}}{|K|},
$$
and similarly for $a_{V, (J_j)}$.

We denote
$$
(S_n^1 S_m^2)^{b, 1,1} f := \sum_{ \substack{K \\  I_i^{(k_i)} = K}} \sum_{ \substack{V \\ J_j^{(v_j)} = V}} \langle b \rangle_{I_1 \times J_1}
a_{K, (I_i)} a_{V, (J_j)} \langle f, h_{I_1} \otimes h_{J_1}\rangle h_{I_2} \otimes h_{J_2}.
$$
So here $1,1$ refers to cubes over which we average $b$ over, namely $\langle b \rangle_{I_1 \times J_1}$. Define $(S_n^1 S_m^2)^{b, 1,2}$ etc. analogously. 
We expand the appearing product in the bi-parameter sense in all of the terms
$$
I = S_n^1 ( b S_m^2 f), \, II = S_n^1 S_m^2 (bf), \, III = b  S_m^2 S_n^1 f\, \textup{ and } \,IV = S_m^2 ( bS_n^1 f). 
$$
The term related to the paraproduct $W(b, \cdot)$
e.g. yields $(S_n^1 S_m^2)^{b, 1,2}$ in $I$, and so we get
\begin{align*}
[S_n^1, [b, &S_m^2]]f = I - II - III + IV = E \\
&+ \sum_{i=1}^8 \big[ S_n^1 ( A_i(b, S_m^2 f)) - S_n^1S_m^2 (A_i(b,f)) - A_i(b,S_m^2 S_n^1 f) + S_m^2 ( A_i(b,S_n^1 f)) \big],
\end{align*}
where
$$
E := (S_n^1 S_m^2)^{b, 1,2} - (S_n^1 S_m^2)^{b, 1,1} - (S_n^1 S_m^2)^{b, 2,2} + (S_n^1 S_m^2)^{b, 2,1}.
$$

We start looking at the sum over $i$. For $i \le 4$ even all the individual terms are bounded.
This is because of Lemma \ref{lem:basicAa} and basic weighted bounds of DMOs:
$$
\|S_n^1 f\|_{L^p(w)} \lesssim_{[w]_{A_p}} \|f\|_{L^p(w)}.
$$
For $i \ge 5$ the term $S_n^1 ( A_i(b, S_m^2 f))$ should be paired with one of the terms with a minus. For $i=5,6$ we pair with
$- A_i(b,S_m^2 S_n^1 f)$ and for $i=7,8$ we pair with $-S_n^1S_m^2 (A_i(b,f))$. Let us take $i=5$. Notice that it is enough to study
$S_n^1 ( A_5(b, f))  - A_5(b,S_n^1 f)$, since $S_n^1S_m^2 = S_m^2 S_n^1$ and $S_m^2$ is bounded. 

Now, the term $S_n^1 ( A_5(b, f))  - A_5(b,S_n^1 f)$ equals
$$
\sum_{ \substack{K \\  I_i^{(k_i)} = K}} \sum_J a_{K, (I_i)} \big[ \langle \langle b, h_J\rangle_2 \rangle_{I_1} - \langle \langle b, h_J\rangle_2 \rangle_{I_2} \big]
\langle f, h_{I_1} \otimes h_J\rangle h_{I_2} \otimes h_J h_J.
$$
Dualising and using that for a function $p$ in $\R^n$ we have $\langle p \rangle_{I_1} - \langle p \rangle_{K} = \sum_{k=1}^{k_1} \langle \Delta_{I_1^{(k)}} p \rangle_{I_1}$, we reduce to estimating, for a fixed $k = 1, \ldots, k_1$, as follows:
\begin{align*}
\sum_{ \substack{K \\  I_i^{(k_i)} = K}} &\sum_J |a_{K, (I_i)}| |I_1^{(k)}|^{-1/2}  |\langle b, h_{I_1^{(k)}} \otimes h_J\rangle| |\langle f, h_{I_1} \otimes h_J\rangle|
\langle |\langle g, h_{I_2} \rangle_1|\rangle_J \\
&= \sum_K \sum_{ \substack{I^{(k_1-k)} = K \\ J}} \sum_{ \substack{I_1^{(k)} =  I \\ I_2^{(k_2)} = K}}
 |a_{K, (I_i)}| |I|^{-1/2}  |\langle b, h_I \otimes h_J\rangle| |\langle f, h_{I_1} \otimes h_J\rangle|
\langle |\langle g, h_{I_2} \rangle_1|\rangle_J \\
&\le  \sum_K \sum_{ \substack{I^{(k_1-k)} = K \\ J}}  |\langle b, h_I \otimes h_J\rangle|
|I|^{1/2} |J|^{1/2} \langle |\Delta_{K \times J}^{k_1, 0} f| \rangle_{I \times J} \langle |g| \rangle_{K \times J} \\
&\lesssim_{[\mu]_{A_p}, [\lambda]_{A_p}} \|b\|_{\BMO_{\textup{prod}}(\nu)} \iint_{\R^{n+m}} \Big( \sum_{K,J} [M_{\calD} \Delta_{K \times J}^{k_1,0} f]^2 \Big)^{1/2} M_{\calD} g \cdot \nu \\
&\lesssim_{[\mu]_{A_p}, [\lambda]_{A_p}}  \|b\|_{\BMO_{\textup{prod}}}(\nu) \|f\|_{L^p(\mu)} \|g\|_{L^{p'}(\lambda^{1-p'})}.
\end{align*}
Besides the various standard weighted estimates, we also used the weighted $H^1$-$\BMO_{\textup{prod}}$ duality estimate:
$$
\sum_{I,J} |\langle b, h_I \otimes h_J\rangle| |c_{I,J}| \lesssim_{[\mu]_{A_p}, [\lambda]_{A_p}} \|b\|_{\BMO_{\textup{prod}}(\nu, \calD)} \iint_{\R^{n+m}}
\Big( \sum_{I,J} |c_{I,J}|^2 \frac{1_I \otimes 1_J}{|I| |J|}\Big)^{1/2} \nu.
$$
For this see Proposition 4.1 in \cite{HPW} (this is what is also used to prove Lemma \ref{lem:basicAa}).
Notice also that there was even too much cancellation in this term (we were able to throw away the cancellative Haar function $h_{I_2}$).
We have shown
$$
\|S_n^1 ( A_5(b, f))  - A_5(b,S_n^1 f)\|_{L^p(\lambda)} \lesssim_{[\mu]_{A_p}, [\lambda]_{A_p}}  k_1\|b\|_{\BMO_{\textup{prod}}}(\nu)\|f\|_{L^p(\mu)}.
$$
The other terms are similar, and so we are done with the $i$ summation. 

It remains to study $E := (S_n^1 S_m^2)^{b, 1,2} - (S_n^1 S_m^2)^{b, 1,1} - (S_n^1 S_m^2)^{b, 2,2} + (S_n^1 S_m^2)^{b, 2,1}$.
After noting that
\begin{align*}
\langle b \rangle_{I_1 \times J_2}& - \langle b \rangle_{I_1 \times J_1} - \langle b \rangle_{I_2 \times J_2}  + \langle b \rangle_{I_2 \times J_1}  =
\sum_{k=1}^{k_1} \sum_{v=1}^{v_2} \langle \Delta_{I_1^{(k)} \times J_2^{(v)}} b \rangle_{I_1 \times J_2} \\
&-\sum_{k=1}^{k_1} \sum_{v=1}^{v_1} \langle \Delta_{I_1^{(k)} \times J_1^{(v)}} b \rangle_{I_1 \times J_1}
-\sum_{k=1}^{k_2} \sum_{v=1}^{v_2} \langle \Delta_{I_2^{(k)} \times J_2^{(v)}} b \rangle_{I_2 \times J_2}
+ \sum_{k=1}^{k_2} \sum_{v=1}^{v_1} \langle \Delta_{I_2^{(k)} \times J_1^{(v)}} b \rangle_{I_2 \times J_1}
\end{align*}
we can reduce the term $E$ to terms, which are very similar to what has been estimated above. Finally, we get
\begin{equation*}\label{eq:SScom}
\|[S_n^1, [b, S_m^2]] \|_{L^p(\mu) \to L^p(\lambda)} \lesssim_{[\mu]_{A_p}, [\lambda]_{A_p}} (1+\max k_i)(1+\max v_i) \|b\|_{\BMO_{\textup{prod}}}(\nu).
\end{equation*}

\subsection*{The case $[\pi_n^1, [b, \pi_m^2]]$}
Here we study the case that
$$
\pi_n^1 f = \sum_K a_K h_K \otimes \langle f \rangle_{K,1}, \,\,
\pi_m^2 f = \sum_V a_V \langle f \rangle_{V,2} \otimes h_V.
$$
Here $K \in \calD^n$ and $V \in \calD^m$, $\|(a_K)\|_{\BMO(\calD^n)} \le 1$ and $a_K \ne 0$ for only finitely many $K$, and similarly for $a_V$.
Paraproducts also have the dual form (unlike shifts which are completely symmetric) -- this is taken into account and we comment on it later on.

For $I = \pi_n^1 ( b \pi_m^2 f)$ we perform a one-parameter expansion in $\R^m$ so that
$$
I = \sum_{i=1}^2 \pi_n^1(a_i^2(b, \pi_m^2 f)) + \pi_n^1(w^2(b, \pi_m^2 f)),
$$
where by adding and subtracting an average we get
$$
\pi_n^1(w^2(b, \pi_m^2 f)) = \sum_{K,V} a_K a_V \langle [ \langle b \rangle_{V,2} - \langle b \rangle_{K \times V} ] \langle f \rangle_{V,2} \rangle_K h_K \otimes h_V
+ (\pi_n^1 \pi_m^2)^b f.
$$
Here we have denoted
$
(\pi_n^1 \pi_m^2)^b f = \sum_{K,V} \langle b \rangle_{K \times V} a_K a_V \langle f \rangle_{K \times V} h_K \times h_V.
$
For the term $II = \pi_n^1 \pi_m^2 (bf)$ we only add and subtract an average to the end that
$$
II = \sum_{K, V} a_K a_V \langle [b - \langle b \rangle_{K \times V}] f \rangle_{K \times V} h_K \otimes h_V + (\pi_n^1 \pi_m^2)^b f.
$$
The term $III = b  \pi_m^2 \pi_n^1 f $ warrants a bi-parameter decomposition so that we get
$$
III = \sum_{i=1}^8 A_i(b, \pi_m^2 \pi_n^1  f) + (\pi_n^1 \pi_m^2)^b f.
$$
Finally, the term $IV = \pi_m^2 ( b\pi_n^1 f)$ is treated symmetrically to $I$ so that
$$
IV = \sum_{i=1}^2 \pi_m^2( a_i^1(b, \pi_n^1 f)) + 
\sum_{K,V} a_K a_V \langle [ \langle b \rangle_{K,1} - \langle b \rangle_{K \times V} ] \langle f \rangle_{K,1} \rangle_V h_K \otimes h_V +  (\pi_n^1 \pi_m^2)^b f.
$$

Now, we get
\begin{equation}\label{eq:parSplit}
\begin{split}
&[\pi_n^1, [b, \pi_m^2]]f = I - II - III + IV \\
&= - \sum_{i=1}^4 A_i(b, \pi_n^1 \pi_m^2 f) 
+ \Big\{ \sum_{i=1}^2 \pi_n^1(a_i^2(b, \pi_m^2 f)) + \sum_{i=1}^2 \pi_m^2( a_i^1(b, \pi_n^1 f)) - \sum_{i=5}^8 A_i(b, \pi_m^2 \pi_n^1  f) \Big\} \\
&+ \sum_{K,V} \Big[ a_K a_V \langle [ \langle b \rangle_{V,2} - \langle b \rangle_{K \times V} ] \langle f \rangle_{V,2} \rangle_K h_K \otimes h_V \\
 & + a_K a_V \langle [ \langle b \rangle_{K,1} - \langle b \rangle_{K \times V} ] \langle f \rangle_{K,1} \rangle_V h_K \otimes h_V
 -  a_K a_V \langle [b - \langle b \rangle_{K \times V}] f \rangle_{K \times V} h_K \otimes h_V\Big].
 \end{split}
\end{equation}
The terms with $A_i$, $i=1,\ldots,4$, are readily in control. The terms inside the bracket are paired in the natural way:
$a_1^2$ goes with $A_5$, $a_2^2$ goes with $A_6$, $a_1^1$ goes with $A_7$ and $a_2^1$ goes with $A_8$.
We study
$$
\pi_n^1(a_1^2(b, \pi_m^2 f)) -  A_5(b, \pi_m^2 \pi_n^1  f) = \pi_n^1(a_1^2(b, \pi_m^2 f)) -  A_5(b, \pi_n^1 \pi_m^2 f).
$$
As usual it is enough to study $\pi_n^1(a_1^2(b, f)) -  A_5(b, \pi_n^1 f)$. We see that it equals
$$
\sum_{K,J} a_K  \langle [ \langle b, h_J \rangle_2 - \langle \langle b, h_J \rangle_2 \rangle_K ] \langle f, h_J \rangle_2 \rangle_K h_K \otimes h_J h_J.
$$
Expanding the product inside the average we see that this further equals
$$
\sum_K a_K \sum_{\substack{ I \subset K  \\ J}} \langle h_I h_I \rangle_K \langle b, h_I \otimes h_J\rangle  \langle f, h_I \otimes h_J \rangle
 h_K \otimes h_J h_J.
$$
Dualising this against a function $g$ we are left with estimating as follows:
\begin{align*}
\sum_K \frac{|a_K|}{|K|}& \sum_{\substack{ I \subset K  \\ J}} |\langle b, h_I \otimes h_J\rangle| |\langle f, h_I \otimes h_J\rangle| 
\langle | \langle g, h_K \rangle_1 | \rangle_J \\
&\lesssim_{[\mu]_{A_p}, [\lambda]_{A_p}} \|b\|_{\BMO_{\textup{prod}}(\nu)} \int_{\R^m} \sum_K |a_K|  M_{\calD^m} \langle g, h_K \rangle_1 \langle S_{\calD}f \cdot \nu\rangle_{K,1} \\
&\lesssim \|b\|_{\BMO_{\textup{prod}}(\nu)} \iint_{\R^n \times \R^m} M_{\calD^n}^1 (S_{\calD}f \cdot \nu) S_{\calD, M}^1 g \\
&\lesssim_{[\mu]_{A_p}, [\lambda]_{A_p}} \|b\|_{\BMO_{\textup{prod}}(\nu)}  \|f\|_{L^p(\mu)} \|g\|_{L^{p'}(\lambda^{1-p'})}.
\end{align*}
The three other natural terms coming from this bracket are estimated in a similar way.

Regarding the last term in \eqref{eq:parSplit} we note that after expanding the products it simply equals
$$
\sum_{K,V} a_K a_V \sum_{ \substack{ I \subset K \\ J \subset V}}
 \langle h_I h_I \rangle_K \langle h_J h_J \rangle_V  \langle b, h_I \otimes h_J\rangle  \langle f, h_I \otimes h_J \rangle
 h_K \otimes h_V.
$$
We dualise and estimate as follows:
\begin{align*}
&\sum_{K,V} \frac{|a_K|}{|K|} \frac{|a_V|}{|V|} |\langle g, h_K \otimes h_V\rangle| \sum_{ \substack{ I \subset K \\ J \subset V}}
|\langle b, h_I \otimes h_J\rangle|   |\langle f, h_I \otimes h_J \rangle| \\
&\lesssim_{[\mu]_{A_p}, [\lambda]_{A_p}} \|b\|_{\BMO_{\textup{prod}}(\nu)} \sum_{K,V} |a_K| |a_V| \langle S_{\calD} f \cdot \nu  \rangle_{K \times V} |\langle g, h_K \otimes h_V\rangle| \\
&\lesssim \|b\|_{\BMO_{\textup{prod}}(\nu)}  \iint_{\R^n \times \R^m} \Big( \sum_{K,V} \langle S_{\calD} f \cdot \nu  \rangle_{K \times V}^2  |\langle g, h_K \otimes h_V\rangle|^2
\frac{1_K \otimes 1_V}{|K||V|} \Big)^{1/2} \\
&\le \|b\|_{\BMO_{\textup{prod}}(\nu)}  \iint_{\R^n \times \R^m} M_{\calD}( S_{\calD} f \cdot \nu) S_{\calD} g \lesssim_{[\mu]_{A_p}, [\lambda]_{A_p}} \|b\|_{\BMO_{\textup{prod}}(\nu)} \|f\|_{L^p(\mu)} \|g\|_{L^{p'}(\lambda^{1-p'})}.
\end{align*}
We have shown
\begin{equation*}\label{eq:PiPicom}
\|[\pi_n^1, [b, \pi_m^2]]\|_{L^p(\mu) \to L^p(\lambda)} \lesssim_{[\mu]_{A_p}, [\lambda]_{A_p}} \|b\|_{\BMO_{\textup{prod}}}(\nu).
\end{equation*}

\subsection*{The case $[S_n^1, [b, \pi_m^2]]$}
Here we consider the mixed case that we have a one-parameter shift and a paraproduct like above.
Expanding $I = S_n^1 ( b \pi_m^2 f)$, $II = S_n^1 \pi_m^2 (bf)$, $III = b  \pi_m^2 S_n^1 f$ and $IV = \pi_m^2 ( bS_n^1 f)$ similarly as above we get
\begin{align*}
&[S_n^1, [b, \pi_m^2]] = I - II - III + IV =  \sum_{i=1}^4 S_n^1(A_i(b, \pi_m^2 f)) - \sum_{i=1}^4 A_i(b, \pi_m^2 S_n^1 f) \\
&+ \Big\{ \sum_{i=5}^8 S_n^1(A_i(b, \pi_m^2 f)) + \sum_{i=1}^2 \pi_m^2(a_i^1(b, S_n^1 f)) - \sum_{i=5}^8 A_i(b, \pi_m^2 S_n^1 f) - \sum_{i=1}^2 S_n^1 \pi_m^2(a_i^1(b,f)) \Big\} \\
&+ \sum_V a_V
 \sum_{ \substack{K \\  I_i^{(k_i)} = K}} \sum_{J \subset V}   \langle h_J h_J \rangle_V a_{K, (I_i)} 
[ \langle \langle b, h_J \rangle_2 \rangle_{I_2} - \langle \langle b, h_J \rangle_2 \rangle_{I_1} ]
 \langle f, h_{I_1} \otimes h_J \rangle   h_{I_2} \otimes h_V.
\end{align*}

We handle the last term first. As usual, we reduce to estimating, for a fixed $k = 1, \ldots, k_1$, as follows:
\begin{align*}
&\sum_V \frac{|a_V|}{|V|} \sum_K \sum_{ \substack{ I^{(k_1-k)} = K \\ J \subset V} } |\langle b, h_I \otimes h_J \rangle | |I|^{-1/2}\sum_{ \substack{ I_1^{(k)} = I \\ I_2^{(k_2)} = K} }
|a_{K, (I_i)}| |\langle f, h_{I_1} \otimes h_J \rangle|  |\langle g, h_{I_2} \otimes h_V \rangle| \\
&\le \sum_V |a_V| |V|^{-1/2} \sum_K \sum_{ \substack{ I^{(k_1-k)} = K \\ J \subset V} } |\langle b, h_I \otimes h_J \rangle | |I|^{1/2} |J|^{1/2} 
\langle |\Delta_J^2 f| \rangle_{I \times J} \langle |\Delta_{K \times V}^{k_2, 0} g| \rangle_{K \times V}.
\end{align*}
This is further dominated in the $ \lesssim_{[\mu]_{A_p}, [\lambda]_{A_p}}$ sense by $\|b\|_{\BMO_{\textup{prod}}(\nu)}$ multiplied with
\begin{align*}
\int_{\R^n} \sum_V &|a_V| |V|^{1/2}  \Big\langle \Big( \sum_J [M_{\calD} \Delta_J^2 f]^2 \Big)^{1/2} \nu\Big\rangle_{V,2} \Big( \sum_K  \langle |\Delta_{K \times V}^{k_2, 0} g| \rangle_{K \times V}^2 1_K \Big)^{1/2} \\
&\lesssim \iint_{\R^n \times \R^m} M_{\calD^m}^2 \Big( \Big( \sum_J [M_{\calD} \Delta_J^2 f]^2 \Big)^{1/2} \nu \Big) \Big( \sum_{K, V}[M_{\calD} \Delta_{K \times V}^{k_2, 0} g]^2\Big)^{1/2} \\
&\lesssim_{[\mu]_{A_p}, [\lambda]_{A_p}}  \|f\|_{L^p(\mu)} \|g\|_{L^{p'}(\lambda^{1-p'})}.
\end{align*}

The sums where $i = 1, \ldots, 4$ are obviously bounded. So it remains to bound the terms inside the bracket. This requires appropriately grouping the terms into pairs of differences.
We simply pair $a_1^1$ with $A_7$, $a_2^1$ with $A_8$, 
and then we pair the remaining $A_5$ terms together and the $A_6$ terms together. This easily (after taking operators as common factors as usually)
reduces to terms that we have already bounded. Therefore, we are done:
\begin{equation*}\label{eq:SPicom}
\|[S_n^1, [b, \pi_m^2]]\|_{L^p(\mu) \to L^p(\lambda)} \lesssim_{[\mu]_{A_p}, [\lambda]_{A_p}} (1+\max k_i)\|b\|_{\BMO_{\textup{prod}}(\nu)}.
\end{equation*}

\subsection*{Rest of the cases}
By duality and symmetry the only case we have to still consider is $[\pi_n^1, [b, \pi_m^2]]$,
where this time one of the paraproducts is in the dual form and one is not:
$$
\pi_n^1 f = \sum_K a_K \frac{1_K}{|K|} \otimes \langle f, h_K \rangle_{1}, \,\,
\pi_m^2 f = \sum_V a_V \langle f \rangle_{V,2} \otimes h_V.
$$
When we decompose $[\pi_n^1, [b, \pi_m^2]]$, the usual terms with the $A_i$, $a_i^1$ and $a_i^2$ produce no surprises.
What is left can be written in the form $E_1 + E_2$, where
\begin{align*}
E_1 &= \pi_n^1 (W(b, \pi_m^2 f)) -\pi_n^1 \pi_m^2 (w^1(b,f)) \\
&= - \sum_{K, V} a_K a_V \langle [ \langle b \rangle_{K,1} - \langle b \rangle_{K \times V} ] \langle f, h_K \rangle_1 \rangle_V \frac{1_K}{|K|} \otimes h_V \\
&= - \sum_{K, V} \sum_{J \subset V}  a_K a_V  \langle h_J h_J \rangle_V \langle f, h_K \otimes h_J\rangle \langle \langle b, h_J\rangle_2 \rangle_K \frac{1_K}{|K|} \otimes h_V
\end{align*}
and
\begin{align*}
E_2  &= -w^2(b, \pi_n^1 \pi_m^2 f) + \pi_m^2(b \pi_n^1f)  \\
&= \sum_{K, V} a_K a_V \langle (b-\langle b \rangle_{V,2} )  \langle f, h_K \rangle_1 \rangle_{V,2}  \frac{1_K}{|K|} \otimes h_V \\
&= \sum_{K, V} \sum_{J \subset V}  a_K a_V \langle h_J h_J \rangle_V \langle f, h_K \otimes h_J\rangle \langle b, h_J\rangle_2 \frac{1_K}{|K|} \otimes h_V.
\end{align*}
Therefore, we have
\begin{align*}
E_1 + E_2 &= \sum_{K, V} \sum_{J \subset V}  a_K a_V \langle h_J h_J \rangle_V \langle f, h_K \otimes h_J\rangle [ \langle b, h_J\rangle_2
- \langle \langle b, h_J\rangle_2 \rangle_K] \frac{1_K}{|K|} \otimes h_V \\
&= \sum_{K, V} \sum_{\substack{ I \subset K \\ J \subset V}} a_K a_V \langle h_J h_J \rangle_V |K|^{-1} 
\langle b, h_I \otimes h_J\rangle  \langle f, h_K \otimes h_J\rangle h_I \otimes h_V.
\end{align*}
We now estimate this by duality, and first arrive at the obvious upper bound
\begin{equation}\label{eq:step1}
\|b\|_{\BMO_{\textup{prod}}(\nu)} \sum_{K,V} \frac{|a_K|}{|K|} \frac{|a_V|}{|V|} \iint_{K \times V}  [S_{\calD^n} \langle g, h_V \rangle_2 \otimes S_{\calD^m} \langle f, h_K \rangle_1]\nu.
\end{equation}
Define the auxiliary functions
$$
\varphi_S^1(f)= \sum_{K} h_K \otimes S_{\calD^m}\langle f, h_K\rangle_1 \, \textup{ and } \,
\varphi_S^2(g)= \sum_{V}  S_{\calD^n}\langle g, h_V\rangle_2 \otimes h_V.
$$
Define also $\wt \pi_n^1 f$ to be the same function as $\pi_n^1 f$ except that $a_K$ is replaced by $|a_K|$.
Define $\wt \pi_m^2 g$ via the formula
$$
\wt \pi_m^2 g =  \sum_V |a_V| \langle g, h_V \rangle_{2} \otimes \frac{1_V}{|V|}.
$$
That is, this is not $\pi_m^2$ where $a_V$ is replaced by $|a_V|$, but rather the dual of $\pi_m^2$ where $a_V$ is replaced by $|a_V|$ (using the
definitions of paraproducts that are valid in this subsection).
Notice that \eqref{eq:step1} equals
\begin{align*}
\|b\|_{\BMO_{\textup{prod}}(\nu)} \iint_{\R^n \times \R^m}& \wt \pi_n^1 (\varphi_S^1(f))  \wt \pi_m^2 (\varphi_S^2(g)) \nu \\
&\le \|b\|_{\BMO_{\textup{prod}}(\nu)} \| \wt \pi_n^1 (\varphi_S^1(f)) \|_{L^p(\mu)}\| \wt \pi_m^2 (\varphi_S^2(g)) \|_{L^{p'}(\lambda^{1-p'})} \\
&\lesssim_{[\mu]_{A_p}, [\lambda]_{A_p}} \|b\|_{\BMO_{\textup{prod}}(\nu)} \|f\|_{L^p(\mu)} \|g\|_{L^{p'}(\lambda^{1-p'})}.
\end{align*}
Here we used the weighted boundedness of the paraproducts and the operators $\varphi_S^i$. For the latter, notice e.g. that
$S_{\calD^n}^1 \varphi_S^1 f = S_{\calD}f$.
We are done with the proof.
\end{proof}

\section{Lower bounds for commutators $[b,\cdots [b, [b, T]]\cdots]$}\label{sec:LowerBounds}
Let $K$ be a standard bi-parameter full kernel as in \cite{Ma1}.
We also assume that
$K$ is uniformly non-degenerate. In our setup this means that for all $y \in \R^{n+m}$ and $r_1, r_2 > 0$ there exists $x \in \R^{n+m}$
such that $|x_1-y_1| > r_1$, $|x_2 - y_2| > r_2$ and
\begin{equation}\label{eq:UND}
|K(x, y)| \gtrsim \frac{1}{r_1^{n} r_2^{m}}.
\end{equation}
For example, we can have 
$$
K(x,y) = K_{i,j}(x,y) = \frac{x_{1,i}-y_{1,i}}{ |x_1-y_1|^{n+1} }\frac{x_{2, j}-y_{2,j} }{ |x_2-y_2|^{m+1} }.
$$
That is, $K = K_{i,j}$ is the full kernel of the bi-parameter Riesz transform $R_i^n \otimes R_j^m$, $i = 1, \ldots, n$, $j = 1, \ldots, m$.
Regarding the assumed H\"older conditions of the kernel $K$, similarly as in \cite{Hy5}, a weaker modulus of continuity should be enough, but we do not pursue this.

We record that \eqref{eq:UND} implies the following: given a rectangle $R = I \times J$ there exists a rectangle $\tilde R = \tilde I \times \tilde J$
such that $\ell(I) = \ell(\tilde I)$, $\ell(J) = \ell(\tilde J)$, $d(I, \tilde I) \sim \ell(I)$, $d(J, \tilde J) \sim \ell(J)$ and such that
for some $\sigma \in \C$ with $|\sigma| = 1$ we have for all $x \in \tilde R$ and $y \in R$
that
$$
\Re \sigma K(x,y) \gtrsim \frac{1}{|R|}.
$$
This can be seen as follows. Let, for a big enough constant $A$, the centre $c_{\tilde R}$ of $\tilde R$ to be the point $x$ given by \eqref{eq:UND}
applied to $y = c_R$, $r_1 = A\ell(I)$ and $r_2 = A\ell(J)$. Choose $\sigma$ so that
$\sigma K(c_{\tilde R}, c_R) = |K(c_{\tilde R}, c_R)|$. Finally, use mixed H\"older and size estimates repeatedly.

Let $k \ge 1$ and $b \in L^k_{\textup{loc}}(\R^{n+m})$ be real-valued.
Let $p > 1$ and $\mu, \lambda \in A_p$. Define $\Gamma = \Gamma(K, b, \mu, \lambda, k, p)$ via the formula
$$
\Gamma = \sup \frac{1}{\mu(R)^{1/p}} \Big\| x \mapsto 1_{\tilde R}(x) \int_A (b(x) - b(y))^{k} K(x,y) \ud y\Big\|_{L^{p, \infty}(\lambda)},
$$
where the supremum is taken over all rectangles $R, \tilde R$ with
$\ell(I) = \ell(\tilde I)$, $\ell(J) = \ell(\tilde J)$, $d(I, \tilde I) \sim \ell(I)$ and $d(J, \tilde J) \sim \ell(J)$, and all subsets $A \subset R$.

The moral is simply the following. Notice that
$$
1_{\tilde R}(x) \int_A (b(x) - b(y))^{k} K(x,y) \ud y = 1_{\tilde R}(x)[b,\cdots [b, [b, T]]\cdots](1_A)(x),
$$
if $T$ is a bi-parameter singular integral with a full kernel $K$. If this iterated commutator maps $L^p(\mu) \to L^{p, \infty}(\lambda)$, then $\Gamma$ is dominated by this norm. However, the constant $\Gamma$ is significantly weaker and depends only on the kernel $K$ and some off-diagonal assumptions.

The following proposition supplies the lower bound related to \cite{LMV2} in the case $b_1 = \cdots = b_k = b$. See Equation \eqref{eq:LMVIterated} in the Introduction.
\begin{prop}
Suppose $K$ is a uniformly non-degenerate bi-parameter full kernel, $k \ge 1$ and $b \in L^k_{\textup{loc}}(\R^{n+m})$ is real-valued.
Let $p > 1$, $\mu, \lambda \in A_p$ and $\nu = \mu^{1/p}\lambda^{-1/p}$. Then for 
$\Gamma = \Gamma(K, b, \mu, \lambda, k, p)$ we have
$$
\|b\|_{\bmo(\nu^{1/k})} \lesssim_{[\mu]_{A_p}, [\lambda]_{A_p}} \Gamma^{1/k}.
$$
\end{prop}
\begin{proof}
Let $R = I \times J$ be a fixed rectangle. As we saw above we can find $\tilde R = \tilde I \times \tilde J$
such that $\ell(I) = \ell(\tilde I)$, $\ell(J) = \ell(\tilde J)$, $d(I, \tilde I) \sim \ell(I)$, $d(J, \tilde J) \sim \ell(J)$ and such that
for some $\sigma \in \C$ with $|\sigma| = 1$ we have for all $x \in \tilde R$ and $y \in R$
that
$$
\Re \sigma K(x,y) \gtrsim \frac{1}{|R|}.
$$
For $t \in \R$ let $t_+ = \max(t,0)$. For an arbitrary $\alpha \in \R$ and $x \in \tilde R \cap \{b \ge \alpha\}$ we have
\begin{align*}
\Big( \frac{1}{|R|} \int_R (\alpha - b)_+ \Big)^k &\le \frac{1}{|R|} \int_{R \cap \{b \le \alpha\}} (b(x) - b(y))^{k}
\ud y \\
& \lesssim \Re \sigma \int_{R \cap \{b \le \alpha\}} (b(x) - b(y))^{k} K(x,y) \ud y.
\end{align*}

Now, let $\alpha$ be a median of $b$ on $\tilde R$ so that
$$
\min \{  |\tilde R \cap \{b\le \alpha\}|, |\tilde R \cap \{b\ge \alpha\}|\}\ge \frac{|\tilde R|}{2} = \frac{|R|}{2}.
$$
In particular, it follows that
$$
\frac{ \lambda( \tilde R \cap \{b\ge \alpha\})}{\lambda(\tilde R)} \ge [\lambda]_{A_p}^{-1} \Big( \frac{ |\tilde R \cap \{b\ge \alpha\}|}{|\tilde R|} \Big)^p \gtrsim_{[\lambda]_{A_p}} 1.
$$
As $\lambda(M R) \lesssim_{M, [\lambda]_{A_p}} \lambda(R)$ we have $\lambda(R) \sim_{[\lambda]_{A_p}} \lambda(\tilde R)$. We then get
\begin{align*}
\lambda(R)^{1/p} \Big( \frac{1}{|R|} \int_R (\alpha - b)_+ \Big)^k &\lesssim_{[\lambda]_{A_p}} \lambda( \tilde R \cap \{b\ge \alpha\})^{1/p} \Big( \frac{1}{|R|} \int_R (\alpha - b)_+ \Big)^k \\
&\lesssim \Big\| x \mapsto 1_{\tilde R}(x) \int_{R \cap \{b \le \alpha\}} (b(x) - b(y))^{k} K(x,y) \ud y\Big\|_{L^{p, \infty}(\lambda)} \\
&\le \Gamma \mu(R)^{1/p}.
\end{align*}
We have 
\begin{align*}
1&=\bla \nu^{\frac 1{k(k+1)}\cdot k}\nu^{-\frac 1{k+1}}\bra_R^{k+1} \\ &\le \bla \nu^{1/k}\bra_R^{ k  }\langle \nu^{-1}\rangle_R
\le  \bla \nu^{1/k}\bra_R^{ k  } \langle \lambda \rangle_R^{1/p}\langle \mu^{1-p'}\rangle_R^{1/p'}\le \bla \nu^{1/k}\bra_R^{ k  }[\mu]_{A_p}^{1/p} \langle \mu \rangle_R^{-1/p} \langle \lambda \rangle_R^{1/p}, 
\end{align*}
and so
$$
\mu(R)^{1/p}\lambda(R)^{-1/p}\lesssim_{[\mu]_{A_p}} \bla \nu^{1/k}\bra_R^{k}.
$$
Combining everything we get
$$
 \Big( \frac{1}{|R|} \int_R (\alpha - b)_+ \Big)^k \lesssim_{[\mu]_{A_p}, [\lambda]_{A_p}} \bla \nu^{1/k}\bra_R^{k} \Gamma.
$$
So we have proved $\int_R (\alpha - b)_+ \lesssim_{[\mu]_{A_p}, [\lambda]_{A_p}} \nu^{1/k}(R) \Gamma^{1/k}$, and the bound
$\int_R (b-\alpha)_+ \lesssim_{[\mu]_{A_p}, [\lambda]_{A_p}} \nu^{1/k}(R) \Gamma^{1/k}$ is proved analogously.
The claim $\|b\|_{\bmo(\nu^{1/k})} \lesssim_{[\mu]_{A_p}, [\lambda]_{A_p}} \Gamma^{1/k}$ follows.

\end{proof}

\end{document}